\documentclass{amsart}
\usepackage{amssymb}
\usepackage[all]{xy}
\usepackage{graphicx}
\usepackage[mathcal]{euscript}
\usepackage{verbatim}

\let\oldmarginpar\marginpar
\renewcommand\marginpar[1]{\-\oldmarginpar[\raggedleft\footnotesize #1]%
{\raggedright\footnotesize #1}}

\renewcommand{\b}{\mathbin{\beta}}
\newcommand{\s}{\sigma}
\newcommand{\z}{\mathbin{\zeta}}
\newcommand{\om}{\omega}
\renewcommand{\l}{\lambda}
\renewcommand{\k}{\kappa}

\newcommand{\bX}{\bar{X}}
\newcommand{\bY}{\bar{Y}}
\newcommand{\bZ}{\bar{Z}}

\newcommand{\bV}{\bar{V}}
\newcommand{\bW}{\bar{W}}

\newcommand{\cC}{\mathcal{C}}

\newcommand{\cH}{\mathcal{H}}

\newcommand{\cP}{\mathcal{P}}
\newcommand{\cQ}{\mathcal{Q}}

\newcommand{\cQG}{\cQ_G}

\newcommand{\C}{\mathbf{C}}
\newcommand{\Z}{\mathbf{Z}}

\DeclareMathOperator{\spa}{span}
\DeclareMathOperator{\End}{End}
\DeclareMathOperator{\tr}{tr}
\DeclareMathOperator{\SU}{SU}
\DeclareMathOperator{\Ind}{Ind}

\DeclareMathOperator{\tchar}{char}

\newtheorem{thm}{Theorem}[section]
\newtheorem{lem}[thm]{Lemma}
\newtheorem{prop}[thm]{Proposition}
\newtheorem{cor}[thm]{Corollary}

\theoremstyle{definition}
\newtheorem{defn}[thm]{Definition}

\theoremstyle{remark}
\newtheorem{rmk}[thm]{Remark}
\newtheorem{exam}[thm]{Example}

\title[Atomistic subsemirings]{Atomistic subsemirings of the lattice
  of subspaces of an algebra} \author{Daniel S.~Sage}
\address{Department of Mathematics\\
  Louisiana State University\\
  Baton Rouge, LA 70803} 
\email{sage@math.lsu.edu}

\thanks{The author gratefully acknowledges support from NSF
grant~DMS-0606300 and NSA grant~H98230-09-1-0059.}

\begin{document}

\begin{abstract}
  Let $A$ be an associative algebra with identity over a field $k$.
  An atomistic subsemiring $R$ of the lattice of subspaces of $A$,
  endowed with the natural product, is a subsemiring which is a closed
  atomistic sublattice.  When $R$ has no zero divisors, the set of
  atoms of $R$ is endowed with a multivalued product.  We introduce an
  equivalence relation on the set of atoms such that the quotient set
  with the induced product is a monoid, called the condensation
  monoid.  Under suitable hypotheses on $R$, we show that this monoid
  is a group and the class of $k1_A$ is the set of atoms of a
  subalgebra of $A$ called the focal subalgebra.  This construction
  can be iterated to obtain higher condensation groups and focal
  subalgebras.  We apply these results to $G$-algebras for $G$ a
  group; in particular, we use them to define new invariants for
  finite-dimensional irreducible projective representations.

\end{abstract}

\maketitle

\section{Introduction}
\label{sect:intro}


Let $A$ be an associative algebra with identity over a field $k$, and
let $S(A)$ be the complete lattice of subspaces of $A$.  The algebra
multiplication on $A$ induces a product on $S(A)$ given by
$EF=\spa\{ef\mid e\in E, f\in F\}$.  The lattice $S$ thus becomes an
additively idempotent semiring, with $\{0\}$ and $k=k1_A$ (which we
will often denote by $0$ and $1$) as the additive and multiplicative
identities.

Let $R$ be a closed sublattice of $S(A)$ which is also a subsemiring,
i.e., $R$ contains $0$ and $k$ and is closed under arbitrary sums and
intersections and finite products.  (We do not require the maximum
element of $R$ to be $A$.)  A nonzero element $X\in R$ is called
decomposable (or join-reducible) if there exists $U,V\subsetneq R$
such that $X=U+V$ and indecomposable otherwise.  It is immediate that
the multiplication in $R$ is determined by the product of
indecomposable elements.  In other words, the semiring structure is
determined by the structure constants $c_{U,V}^W$ for $U,V,W\in R$
indecomposable, where $c_{U,V}^W$ is $1$ if $W\subset UV$ and $0$
otherwise.

In this paper, we consider subsemirings $R$ whose product is
determined by its minimal nonzero elements--the atoms of the lattice.
This means that the indecomposable elements of $R$ are precisely the
atoms, so that every nonzero element is a join of atoms, i.e., $R$ is
an atomistic lattice\footnote{In the usual definition, every nonzero
  element of an atomistic lattice is a finite join of atoms.  In this
  paper, we allow arbitrary joins of atoms.}.

\begin{defn} A subsemiring $R\subset S(A)$ is called an \emph{atomistic
    subsemiring} of $S(A)$ if it is also a closed atomistic sublattice.
\end{defn}

Note that $k$ is always an atom in $R$.

\begin{exam} For any $A$, $S(A)$ and $\{0,1\}$
    are atomistic subsemirings.
  \end{exam}
\begin{exam} Let $X$ be any proper subspace with $X+ k1_A=A$.  Then
  $R=\{0,1,X,A\}$ is an atomistic subsemiring if and only if $X^2\in
  R$.  All four possible values for $X^2$ can occur.  Indeed, if we
  let $X=k\bar{t}$ in the three two-dimensional algebras
  $k[t]/(t^2)$, $k[t]/(t^2-1)$, and $k[t]/(t^2-t)$, we obtain $X^2$
  equal to $0$, $1$, and $X$ respectively.  On the other hand, if
  $X=\spa(\bar{t},\bar{t}^2)$ in $A=k[t]/(t^3-1)$, then $X^2=A$.
  (Note that there are never any atomistic subsemirings of size $3$.)
\end{exam}
\begin{exam} Let $V$ be a vector space with $\dim V\ge 2$, and suppose
  $(\tchar{k},\dim V)=1$. Let $A=\End(V)$, and let
  $X=\{x\in\End(V)\mid \tr(x)=0\}$.  Then $R=\{0,1,X,A\}$ is atomistic
  with $X^2=A$.  To see this, simply note that every matrix unit lies
  in $X^2$: $E_{ii}=E_{ij}E_{ji}$ and $E_{ij}=E_{ij}(E_{ii}-E_{jj})$
  where $i\ne j$.
\end{exam}

Our primary motivation for considering atomistic subsemirings comes
from representation theory.  Let $G$ be a group which acts on $A$ by
algebra automorphisms.  This means that $A$ is a $k[G]$-module such
that $g\cdot 1_A=1_A$ and $g\cdot(ab)=(g\cdot a)(g\cdot b)$ for all
$g\in G$ and $a,b\in A$.  We let $S_G(A)\subset S(A)$ be the set of
all $k[G]$-submodules of $A$.  This set, called the
\emph{subrepresentation semiring} of $A$, is simultaneously a
subsemiring and complete sublattice of $S(A)$; such semirings were
introduced and studied in \cite{S2,S3}.  If $A$ is a completely
reducible representation, i.e., a direct sum of irreducible
representations, then $S_G(A)$ is an atomistic subsemiring.  For
example, this occurs when $G$ is finite, $A$ is finite-dimensional,
and $k$ has characteristic zero.

When $G=\SU(2)$ (or more generally, $G$ is a \emph{quasi-simply
  reducible} group), then the subrepresentation semirings for the
$G$-algebras $\End(V)$ (with $V$ a representation of $G$) have had
important applications in materials science and
physics~\cite{GS,GMS,S2}.  The structure of such semirings is
intimately related to the theory of $6j$-coefficients from the quantum
theory of angular momentum~\cite{S2,S3,S4,KS}.

Our goal in this paper is to study the set of atoms $\cQ(R)$ of an
atomistic subsemiring and to use it to define new invariants for
appropriate $R$--the condensation group, the focus, the focal
subalgebra, and higher analogues.  Our methods are motivated by the
theory of hypergroups.

We now give a brief outline of the
contents of the paper.  In Section~\ref{S:atom}, we define a
multivalued product on the set $\cQ(R)$ of atoms of an atomistic
subsemiring $R$.  In the next section, we introduce an equivalence
relation $\z^*$ on $\cQ(R)$.  We show that if $R$ has no
zero-divisors, then the quotient set $\cQ(R)/\z^*$ is naturally a
monoid (called the condensation monoid) while if $R$ is \emph{weakly
  reproducible}, the condensation monoid is in fact a group.  In
Section~\ref{S:focus}, we define the focus $\varpi_R\subset\cQ(R)$ and
focal subalgebra $F(R)\subset A$ of $R$.  The main result is
Theorem~\ref{T:focus}, which states that if $R$ is weakly reproducible
of finite length, then $[0,F(R)]$ is an atomistic subsemiring with the
same properties and whose set of atoms is $\varpi_R$.  This allows us
to iterate our construction to obtain higher order versions of our
invariants.  In Section~\ref{S:complete}, we prove
Theorem~\ref{T:focus} by analyzing \emph{complete subsets} of
$\cQ(R)$.  We apply our results to $G$-algebras in the final section.
In particular, we show how to associate new invariants to irreducible
projective representations.

\section{A hyperproduct on the set of atoms}\label{S:atom}

From now on, $R$ will always be an atomistic subsemiring of $S(A)$.
Let $\cQ(R)$ denote the set of atoms of $R$.  If $R=S_G(A)$ for a
$G$-algebra $A$, we write $\cQ_G(A)$ instead of $\cQ(S_G(A))$.  We
make the notational convention that, unless otherwise specified,
capital letters towards the end of the alphabet will denote atoms.

There is a natural operation $\cQ(R)\times\cQ(R)\to\cP(\cQ(R))$ given
by $X\circ Y=\{Z\in \cQ(R)\mid Z\subset XY\}$.  Our first goal is to
find a natural equivalence relation on $\cQ(R)$ (for appropriate $R$)
for which $\circ$ induces a monoid (or group) structure on the set of
equivalence classes.

Before proceeding, we need to recall some definitions from the theory
of hypergroups.  A set $\cH$ is called a hypergroupoid if it is
endowed with a binary operation $\circ:\cH\times\cH\to\cP^*(\cH)$,
where $\cP^*(\cH)$ is the set of nonempty subsets of $\cH$.  If this
operation is associative, then $\cH$ is called a semihypergroup; if
$\cH$ also satisfies the reproductive law $\cH\circ x=\cH=x\circ \cH$
for all $x\in\cH$, then $\cH$ is called a hypergroup.  (For more
details on hypergroups, see the books by Corsini~\cite{Cor} and Vougiouklis~\cite{Vou}.)

An element $e$ of the hypergroupoid $\cH$ is called a scalar identity
if $e\circ x=\{x\}=x\circ e$ for all $x\in\cH$; if a scalar identity
exists, it is unique.  For later use, we introduce a weak version of
the reproductive law.  A hypergroupoid with scalar identity $e$
satisfies the \emph{weak reproductive law} if for any $x\in\cH$, there
exists $u,v\in\cH$ such that $e\in x\circ u\cap v\circ x$.  Note that
a semihypergroup that satisfies the weak reproductive law is a
hypergroup.  Indeed, given $y\in\cH$, $y\in y\circ e\subset y\circ
(v\circ x)=(y\circ v)\circ x$, so there exists $w\in y\circ v$ such
that $y\in w\circ x$.  Similarly, there exists $w'$ such that $y\in
x\circ w'$.

In general, $\cQ(R)$ is not even a hypergroupoid.  However, we have
the following result:
\begin{prop} Let $R$ be an atomistic subsemiring.  Then $(\cQ(R),\circ)$
  is a hypergroupoid if and only if $R$ is an entire semiring (i.e.,
  $R$ has no left or right zero divisors).
\end{prop}
\begin{proof} Suppose $R$ is entire.  If $X,Y\in \cQ(R)$, then the
  nonzero subspace $XY$ must contain an atom, so $X\circ Y\ne
  \varnothing$.  Conversely, if $E,F$ are nonzero elements of $R$ such
  that $EF=0$, then choosing $X,Y\in\cQ(R)$ such that $X\subset E$ and
  $Y\subset F$ implies that $XY=0$, i.e., $X\circ Y=\varnothing$.
\end{proof}

In particular, if $A$ has zero divisors, then $\cQ(S(A))$ is not a
hypergroupoid.  We will only be interested in atomistic subsemirings
$R$ for which $\cQ(R)$ is a hypergroupoid, so, from now on, we assume
that $R$ is entire, unless otherwise specified.  Note that $k$ is a
scalar identity for $\cQ(R)$.

We begin by considering a motivating example.  We need to recall some
basic properties of semisimple, multiplicity-free representations.
This class of $G$-modules is closed under taking submodules and
quotients.  Any such representation $V$ is the direct sum of its
irreducible submodules, and this is the only way of decomposing $V$ as
the internal direct sum of irreducible submodules.  Moreover, there is
a bijection between the power set of the set of irreducible submodules
of $V$ and the set of subrepresentations of $V$ given by $J\mapsto
\sum_{X\in J} X$.  It follows that if $\{V_i\mid i\in I\}$ is a
collection of submodules of $V$ and $W=\sum_{i\in I} V_i$, then for
$X$ irreducible, $X\subset W$ if and only if $X\subset V_j$ for some
$j\in I$.

\begin{prop}\label{Prop:mf} Let $A$ be a multiplicity-free $G$-algebra with no
  proper, nontrivial left (or right) invariant ideals.  Then $\cQ_G(A)$ is
  a hypergroup.
\end{prop}

\begin{proof}  First, we show that the multiplication on $\cQ_G(A)$ is
  associative.  Fix  $X,Y,Z\in \cQ_G(A)$.  Since $A$ is
  multiplicity-free, $XY=\sum_{j\in J}U_j$, where $X\circ Y=\{U_j\mid
  j\in J\}$.  As discussed above, an irreducible submodule $W$ lies in
  $(XY)Z=\sum U_jZ$ if and only if it is contained in $U_iZ$ for some
  $i$, i.e., $W\in U_i\circ Z$.  We thus see that $(X\circ Y)\circ Z$
  is the set of irreducible submodules of $XYZ$.  A similar argument
  shows that the same holds for $X\circ(Y\circ Z)$.

  Next, we show that $X\circ Y\ne \varnothing$ for any $X,Y\in
  \cQ_G(A)$.  It suffices to show that $XY\ne 0$ for all $X,Y$.  Let
  $Y^\perp=\{a\in A\mid ay=0\text{ for all }y\in Y\}$.  The subspace
  $Y^\perp$ is clearly a left ideal.  Moreover, it is a
  subrepresentation: given $g\in G, u\in Y^\perp$, $(g\cdot
  a)u=g\cdot(a(g^{-1}\cdot u))=g\cdot 0=0$.  Since $Y^\perp\ne A$, our
  hypothesis on invariant left ideals implies that $Y^\perp=0$ and
  $XY\ne 0$ for all $X$. 

  Finally, we show that $X\circ \cQ_G(A)=\cQ_G(A)=\cQ_G(A)\circ X$ for
  any $X$.  The subspace $AX$ is a nonzero left ideal which is
  obviously a subrepresentation, so $AX=A$.  Writing $A$ as a sum of
  irreducible submodules $A=\sum_{i\in I} U_i$, we have $A=\sum U_iX$.
  The usual multiplicity-free argument shows that each $U_j$ lies in
  some $U_{i_j}X$, so $U_j\in U_{i_j}\circ X$.  The other equality
  uses the condition on invariant right ideals.
\end{proof}

Matrix algebras give an important class of examples.  If $V$ is a
finite-dimensional vector space and $\End(V)$ is a $G$-algebra, then
$V$ is naturally a projective representation of $G$~\cite{S1}.  It was
further shown in \cite{S1} that $\End(V)$ for such representations has
no proper, nontrivial invariant left or right ideals if and only if
$V$ is irreducible.  Hence, we obtain:

\begin{cor} If $V$ is a finite-dimensional irreducible projective
  representation of a group such that $\End(V)$ is multiplicity free,
  then $\cQ_G(\End(V))$ is a hypergroup.
\end{cor}

This corollary applies, for example, to every irreducible
complex representation of $\SU(2)$.

The importance of Proposition~\ref{Prop:mf} stems from the fact that
there is a group naturally associated to every hypergroup.  More
generally, let $\cH$ be a semihypergroup.  Consider the relation $\b$
defined by $x\b y$ if and only if there exists $z_1,\dots,z_n\in\cH$
such that $x,y\in z_1\circ\dots\circ z_n$.  Koskas showed that if
$\b^*$ is the transitive closure of $\b$, then the induced
multiplication makes $\cH/\b^*$ into a semigroup, and $\b^*$ is the
largest equivalence relation on $\cH$ with this
property~\cite{Koskas}.  If $\cH$ is a hypergroup, then Freni proved
that $\b$ is automatically transitive~\cite{Freni91}; thus, $\cH/\b$
is a group.

We are led to the following provisional definition.

\begin{defn} Let $A$ be a $G$-algebra satisfying the hypotheses of
  Proposition~\ref{Prop:mf}.  The group $Q_G(A)=\cQ_G(A)/\b$ is called
  the \emph{condensation group}  of $A$.
\end{defn}

We will generalize this definition to a much broader class of
atomistic subsemirings below.  However, before continuing we provide a
few examples.

\begin{exam} If $k$ denotes the trivial $G$-algebra, then $Q_G(k)$ is the trivial group.
\end{exam}
\begin{exam} If $V$ is any irreducible representation of $\SU(2)$,
  then $Q_{\SU(2)}(\End(V))=1$.  The proof is a special case of
  Theorem~\ref{T:Lie} below.
\end{exam}
\begin{exam} Let $V$ be the standard representation of $S_3$ over the
  complex numbers.  The corresponding $S_3$-algebra decomposes as
  $\End(V)=\C\oplus \s\oplus V$, where $\s$ is the sign
  representation.  Since $\s^2=\C$, $\s V=V\s=V$, and
  $V^2=\C\oplus\s$, we see that the classes of $\b$ are $\{\C,\s\}$
  and $\{V\}$; hence, the condensation group has order $2$.
\end{exam}
\begin{exam} If $F$ is a finite Galois extension of $k$ with abelian
  Galois group $G$, then $\cQG(F)=G$.
\end{exam}

We remark that if the relation $\b$ is replaced by Freni's relation
$\gamma$~\cite{Freni02}, one gets an abelian group canonically related
to any hypergroup.  However, we will not attempt to generalize the
abelian group $\cQ_G(A)/\gamma$ to other atomistic subsemirings in
this paper.

\section{The equivalence relation $\z^*$}\label{S:zeta}

It is not true in general that the hypergroupoid $\cQ(R)$ is a
hypergroup or even a semihypergroup.  For example, the binary
operation on $\cQ_{A_4}(\End(W))$ is not associative, where $W$ is the
three-dimensional irreducible representation of $A_4$.  Moreover, the
reproductive law is not satisfied.  (See Example~\ref{E:A4} below.)
We can thus no longer use the relation $\b^*$ to associate a monoid or
group to $R$.  Instead, we will do so by introducing a new relation
$\z$.  This relation will coincide with $\b$ in the situation of
Proposition~\ref{Prop:mf}.

\begin{defn}  The relation $\z$ on $\cQ(R)$ is defined by $X\z Y$ if and
  only if there exists $Z_1,\dots,Z_n\in\cQ(R)$ such that $X,Y\subset
  \prod_{i=1}^n Z_i$.  We let $\z^*$ denote the transitive closure of $\z$.
\end{defn}

It is obvious that $\z^*$ is an equivalence relation.  We will let
$\bX$ denote the equivalence class of $X\in\cQ(R)$.

\begin{rmk}   If $Z$ is an atom contained in the $\circ$ product of
  $Z_1,\dots,Z_n$ with any choice of parentheses, then $Z\subset
  \prod_{i=1}^n Z_i$.  In fact, the relation $\b$ can be defined for
  hypergroupoids, and this observation just says that $\b\subset\z$.
  However, the set of $\b^*$-equivalence classes is not necessarily a monoid.
\end{rmk}

\begin{defn}Let $R$ be an entire, atomistic subsemiring of
    $S(A)$.
\begin{enumerate}\item  $R$ is called \emph{weakly reproducible} if
  the hypergroupoid $\cQ(R)$ satisfies the weak reproductive law, i.e.,
  for all $X\in\cQ(R)$, there exists $Y,Z\in\cQ(R)$ such that $k\in
  X\circ Y\cap Z\circ X$.
\item $R$ is called \emph{reproducible} if $\cQ(R)$ satisfies the
  reproductive law, i.e., for all $X\in\cQ(R)$, $\cQ(R)\circ
  X=\cQ(R)=X\circ \cQ(R)$.
\end{enumerate}
\end{defn}
\begin{rmk}\label{R:entire} One can define an atomistic subsemiring $R$ to be weakly
  reproducible without the assumption that $R$ is entire.  However,
  $R$ is then entire automatically.  Indeed, if $XY=0$ for
  $X,Y\in\cQ(R)$, then weak reproducibility implies the existence of
  $Z$ such that $k\subset ZX$, so $Y=kY\subset ZXY=0$, a
  contradiction.
\end{rmk}

\begin{thm} Let $R$ be an entire, atomistic semiring of $S(A)$.  Then
\begin{enumerate} \item The induced multiplication on classes makes
  $Q(R)\overset{\text{def}}{=}\cQ(R)/\z^*$ into a monoid.  
\item If $R$ is weakly reproducible, then $Q(R)$ is a group.
\end{enumerate}
\end{thm}

\begin{defn}  The monoid $Q(R)$ is called the \emph{condensation
    monoid} (or \emph{group}) of $R$.
\end{defn}

The following lemma shows that this terminology does not conflict with
our previous definition.

\begin{lem} If $A$ satisfies the hypotheses of
  Proposition~\ref{Prop:mf}, then $\b$ and $\z$ coincide on $\cQ_G(A)$.
\end{lem}
\begin{proof} A similar argument to that used to demonstrate the
  associativity of $\cQ_G(A)$ shows that $Z_1\circ\dots\circ Z_n$ is
  the set of irreducible submodules of $\prod_{i=1}^n Z_i$, so
  $\b=\z$.
\end{proof}

Recall that an equivalence relation $\mathbin{\sim}$ on a
hypergroupoid $\cH$ is called strongly regular if, for any $x,y$ such
that $z\mathbin{\sim}y$ and any $w\in\cH$, then $u\in x\circ w$ and
$v\in y\circ w$ (resp.  $u\in w\circ x$ and $v\in w\circ y$) implies
that $u\mathbin{\sim} v$.  It is a standard fact that for such
$\mathbin{\sim}$, $\circ$ induces a binary operation on
$\cH/\mathbin{\sim}$ via $\bar{x}\circ\bar{y}=\bar{z}$, where $z\in
x\circ y$ \cite{Cor}.  Indeed, strong regularity implies that the set
$\{\bar{z}\mid z\in x'\circ y'\text{ for some }x'\in\bar{x},
y'\in\bar{y}\}$ is a singleton.

\begin{lem} The equivalence relation $\z^*$ is strongly regular.
\end{lem}
\begin{proof} First, suppose that $X\z Y$, so $X,Y\subset\prod_{i=1}^n
  Z_i$ for some $Z_i$'s.  If $U\in X\circ W$ and $V\in Y\circ W$, then
  $U\subset XW$ and $V\subset YW$.  Thus, $U,V\subset (\prod_{i=1}^n
  Z_i)W$, i.e., $U\z W$.  If $X\z^* Y$, then there exists
  $X_0,\dots,X_s\in\cQ(R)$ with $X=X_0$, $Y=X_s$, and $X_i\z X_{i+1}$
  for all $i$.  Taking $U_i\in X_i\circ W$  with
  $U=U_0$ and $V=U_s$, the previous case shows that $U_i\z U_{i+1}$
  for all $i$, i.e., $U\z^* V$.  The opposite direction in the
  definition of strong regularity is proved similarly.
\end{proof}

We now verify that the induced binary operation makes $Q(R)$ into a
monoid.  The identity is given by $\bar{k}$; indeed, this follows
immediately from the fact that $k\circ X=X=X\circ k$.  Next, we check
that $(\bX\circ\bY)\circ\bZ=\bX\circ(\bY\circ\bZ)$.  Choose $U\in
X\circ Y$ and $V\in U\circ Z$, so that $\bV=(\bX\circ\bY)\circ\bZ$.
Since $U\subset XY$, $V\subset UZ\subset XYZ$.  Similarly, choosing
$T\in Y\circ Z$ and $W\in X\circ T$ gives $\bW=\bX\circ(\bY\circ\bZ)$
and $W\subset XT\subset XYZ$.  By definition, $V\z W$, so $Q(R)$ is
associative.

\begin{rmk}  If we allow $R$ to be an atomistic hemiring of $S(A)$,
  i.e., we do not require that $k\in R$, then the same argument shows
  that $Q(R)$ is a semigroup.
\end{rmk}

Finally, assume that $R$ is weakly reproducible.  Given $X\in \cQ(R)$,
choose $Y,Z$ such that $k\subset XY\cap ZX$.  By definition of the
product on $Q(R)$, we obtain $\bX\circ\bY=\bar{k}=\bZ\circ \bX$, so
$\bX$ is left and right invertible.  This shows that $Q(R)$ is a group
and finishes the proof of the theorem.

\begin{rmk} Any monoid can be realized as the condensation monoid of
  an atomistic subsemiring.  Indeed, given a monoid $M$, let $kM$ be
  the corresponding monoid algebra over $k$ with basis elements
  $\{e_x|x\in M\}$.  Let $R=\{\spa\{e_x\mid x\in F\}\mid F\subset
  M\}$.  This is an entire atomistic subsemiring of $S(kM)$ with
  $\cQ(R)=\{ke_x\mid x\in M\}$.  It is now easy to see that $Q(R)=M$.
\end{rmk}

\section{The focus and the focal subalgebra}\label{S:focus}

Recall that if $\cH$ is a hypergroup, the \emph{heart} $\om_{\cH}$ of
$\cH$ is the kernel of the canonical homomorphism
$\phi:\cH\to\cH/\b^*$; it is a subhypergroup of $\cH$.  Returning to
the context of Proposition~\ref{Prop:mf}, let $A$ be a
multiplicity-free $G$-algebra with no proper, nonzero left or right
invariant ideals.  We may then use the heart $\om$ of the hypergroup
$\cQ_G(A)$ to define an invariant subalgebra with the same properties.

\begin{prop} Let $A$ be a multiplicity-free $G$-algebra with no
  proper, nontrivial one-sided invariant ideals.  Then
  $B=\sum \{X\mid X\in\om\}$ is a multiplicity-free $G$-subalgebra with no
  proper one-sided invariant ideals.
\end{prop}
\begin{proof} It is trivial that $B$ is a multiplicity-free
  subrepresentation that contains $k$.  Moreover, if $X,Y\in\om$
  and $Z\subset XY$ is irreducible, then $\phi(Z)=\phi(X)\phi(Y)=1$,
  i.e., $Z\in\om$.  This means that $Z$ and hence $XY$ are subspaces
  of $B$.  It remains to show that $BX=B=XB$ for any $X\in\om$.
  Choose $Z\in\om$.  Since $\cQ_G(A)$ is a hypergroup, there exists
  $Y$ irreducible such that $Z\in Y\circ X$.  Since
  $1=\phi(Z)=\phi(Y)\phi(X)=\phi(Y)$, we see that $Y\in\om$, so
  $Z\subset BX$.  The proof that $Z\subset XB$ is similar.
\end{proof}

This result allows us to iterate the construction of the condensation
group.  Indeed, the hypergroup structure on $\cQ_G(B)=\om$ gives rise
to the group $Q_G(B)$ and an invariant subalgebra $B'\subset B$ such
that $\cQ_G(B')$ is again a hypergroup.  See Section~\ref{S:Galgs} for
examples.

Motivated by this situation, we make the following definitions.
\begin{defn} Let $R$ be an entire atomistic subsemiring.
  \begin{enumerate} \item The \emph{focus} $\varpi_R$ of $R$ is the
    kernel of the homomorphism $\psi_R:\cQ(R)\to Q(R)$.  Equivalently,
    it is the equivalence class of $k$.
\item The subspace $F(R)=\sum \{X\mid X\in\varpi_R\}\in R$ is called the
  \emph{focal subalgebra} associated to $R$.
\end{enumerate}
\end{defn}

We can now state one of the main results of the paper.

\begin{thm} \label{T:focus} Let $R$ be an entire atomistic subsemiring of $S(A)$.
\begin{enumerate}\item \label{part1}The focal subspace $F(R)$ is a unital subalgebra of $A$.
\item \label{part2} The sublattice $[0,F(R)]\subset R$ is an entire
  atomistic subsemiring of $S(F(A))$ with $\varpi_R\subset
  \cQ([0,F(R)])$.
\item \label{repro} If $R$ is weakly reproducible and has finite length, then
  $\cQ([0,F(R)])=\varpi_R$.
\item \label{part4} If $R$ is weakly reproducible (resp. reproducible)
  of finite length, then the same holds for $[0,F(R)]$.
\end{enumerate}
\end{thm}
We remark that part~\eqref{repro} is very useful in computations as it
is often easier to calculate $F(R)$ than to compute $\varpi_R$
directly.

We will only prove the first two parts of the theorem now.  The proof
of the other parts requires a more detailed study of the relation
$\z^*$ and will be given at the end of Section~\ref{S:complete}.

\begin{proof}[Proof of parts~\eqref{part1} and \eqref{part2}]

  If $X,Y\in\varpi_R$ and $Z\in X\circ Y$, then
  $1=\psi(X)\psi(Y)=\psi(Z)$.  This means that $Z\in\varpi$, so
  $XY\subset F(R)$.  Since $k\subset F(R)$, $F(R)$ is a subalgebra.
  This implies that $F(R)^2=F(R)$, so if $E,E'\in [0,F(R)]$, then
  $E+E'\subset F(R)$ and $EE'\subset F(R)$.  Thus, the closed
  sublattice $[0,F(R)]\subset R$ is a subsemiring of $R$, and it is
  immediate that it is entire and atomistic.  The atoms of $[0,F(R)]$
  are precisely the atoms of $R$ which are contained in $F(R)$, so
  $\varpi_R\subset \cQ([0,F(R)])$.
\end{proof}

\begin{cor} \label{C:max} If $R$ is weakly reproducible and has finite
  length, then $Q(R)=1$ if and only if $F(R)$ is the maximum element
  of $R$, i.e., $[0,F(R)]=R$.
\end{cor}
\begin{proof}  If $Q(R)=1$, then $\varpi_R=\cQ(R)$.  Thus,
  $F(R)$ contains every atom in $R$, hence is the maximum element of
  $R$.  Conversely, if $F(R)$ is the maximum of $R$, then
  part~\eqref{repro} of the theorem
  implies that $\varpi_R=\cQ(R)$.  This gives $Q(R)=1$.
\end{proof}
\begin{rmk} The forward implication in the corollary holds for any
  entire atomistic subsemiring.
\end{rmk}
 
The theorem shows that we can iterate the construction of the
invariants associated to $R$.
\begin{defn} The higher foci, focal subalgebras, and condensation
  monoids (or groups) for $R$ are defined recursively as follows:
\begin{itemize} \item $\varpi^1_R=\varpi_R$, $F^1(R)=F(R)$, and
  $Q^1(R)=Q(R)$;
\item $\varpi^{n+1}_R=\varpi_{[0,F^n(R)]}$, $F^{n+1}(R)=F([0,F^n(R)])$,
  and $Q^{n+1}(R)=Q([0,F^n(R)])$.
\end{itemize}
\end{defn}

We observe that if $R$ is weakly reproducible and has finite length,
then $Q^n(R)$ is a group for all $n$.

\section{Complete subsets of $\cQ(R)$}\label{S:complete}

In order to prove Theorem~\ref{T:focus}, we need a better
understanding of the equivalence relation $\z^*$.  In this section, we
define \emph{complete subsets} of $\cQ(R)$ and use them to investigate
the $\z^*$-equivalence classes.  Our analysis of $\z^*$ follows a
similar pattern to that of $\b^*$ carried out by Corsini and
Freni~\cite{Cor,Freni91}.  In the end, we will show that if $R$ is weakly
reproducible, then every element of $\varpi_R$ is $\z$-related (and
not just $\z^*$-related) to $k$; this will be the key ingredient in
the proof of Theorem~\ref{T:focus}.

\begin{defn}\mbox{}  \begin{enumerate} \item A subset $E\subset
    \cQ(R)$ is called \emph{complete} if for all $X_1,\dots, X_n\in
    \cQ(R)$, if there exists $X\in E$ such that $X\subset\prod_{i=1}^n
    X_i$, then for any $Y\subset\prod_{i=1}^n X_i$, $Y\in E$.
\item If $E$ is a nonempty subset of $\cQ(R)$, then the intersection
  of all complete subsets containing $E$ is denoted by $\cC(E)$; it is
  called the \emph{complete closure} of $E$.
\end{enumerate}
\end{defn}

It is obvious that $\cC(E)$ is the smallest closed subset containing $E$.

\begin{rmk}  This is not the usual notion of a complete subset of a
  semihypergroup~\cite{Cor,Koskas}, though it coincides in the context of
  Proposition~\ref{Prop:mf}.  In this paper, we only consider completeness in
  the sense given above.
\end{rmk}

The basic examples of closed subsets are the $\z^*$-equivalence classes.

\begin{prop}  Any $\z^*$-equivalence class is closed.
\end{prop}
\begin{proof}  Consider the class of $Z$.  Suppose that $X\z^* Z$ and
  $X,Y\subset\prod_{i=1}^n X_i$.  Then $Y\z X$, so $Y\z^* Z$.
\end{proof}

The complete closure may be computed inductively.  Indeed, given
$E\ne\varnothing$, define a sequence of subsets $\k_n(E)\subset\cQ(R)$
recursively as follows:  $\k_1(E)=E$ and 
\begin{equation*} \k_{n+1}(E)=\{X\in\cQ(R)\mid \exists
  Y_1,\dots,Y_s\in\cQ(R)\text{ and } Y\in \k_n(E)\text{ such
    that } X,Y\subset \prod_{i=1}^s Y_i\}.
\end{equation*}
Set $\k(E)=\cup_{n\ge 1}\k_n(E)$.

\begin{prop}\label{P:closure} For any nonempty $E\subset \cQ(R)$, $\cC(E)=\k(E)$.
\end{prop}
\begin{proof} Suppose $Y\in\k(E)$, say $Y\in \k_n(E)$, and $X,Y\subset
  \prod_{i=1}^s Y_i$.  Then $X\in\k_{n+1}(E)\subset\k(E)$, so $\k(E)$
  is complete.  Since $E\subset\k(E)$, $\cC(E)\subset\k(E)$.
  Conversely, suppose that $F\supset E$ and $F$ is complete.  We show
  inductively that $\k_n(E)\subset F$.  This is obvious for $n=1$.
  Suppose $\k_n(E)\subset F$.  If $X\in\k_{n+1}(E)$, then we can find
  $Y_1,\dots, Y_s\in\cQ(R)$ and $Y\in\k_n(E)$ such that $X,Y\subset
  \prod_{i=1}^s Y_i$.  Completeness of $F$ now shows that $X\in F$ as
  desired.
\end{proof}

We can now give a new characterization of $\z^*$.  Define a relation
$\mathbin{\k}$ on $\cQ(R)$ by $X\mathbin{\k} Y$ if and only if $X\in
\cC(Y)$, where $\cC(Y)=\cC(\{Y\})$.

\begin{thm}\label{T:kappa} The relations $\mathbin{\k}$ and $\z^*$ coincide.
\end{thm}

Before beginning the proof, we will need a lemma.

\begin{lem}\mbox{}\begin{enumerate} \item \label{kk} For any $X\in\cQ(R)$ and $n\ge 2$,
    $\k_{n+1}(E)=\k_n(\k_2(X))$.
\item For $X,Y\in\cQ(R)$, $X\in \k_n(Y)$ if and only if $Y\in \k_n(X)$.
\end{enumerate}
\end{lem}
\begin{proof}  

  Note that $\k_{n}(\k_2(X))$ consists of those atoms $Z$ for which
  there exists $Y_i$'s and $Y\in\k_{n-1}(\k_2(X))$ such that
  $Y,Z\subset \prod_{i=1}^s Y_i$.  If $n=2$, then
  $\k_{n-1}(\k_2(X))=\k_2(X)$, and this is precisely the defining
  property of $\k_3(X)$.  If $n>2$, then $\k_{n-1}(\k_2(X))=\k_n(X)$
  by inductive hypothesis, and we see that such atoms are precisely
  the elements of $\k_{n+1}(X)$.  This proves part~\eqref{kk}.

  The second assertion is also proven by induction.  Suppose
  $X\in\k_2(Y)$.  Then there exist $Y_i$'s such that
  $X,Y\subset\prod_{i=1}^s Y_i$, so $Y\in\k_2(X)$.  Next, assume that
  the statement holds for $n$.  If $X\in\k_{n+1}(Y)$, then $X,Z\subset
  \prod_{i=1}^s Y_i$ for some $Y_i$'s and $Z\in\k_n(Y)$.  By
  definition, $Z\in\k_2(X)$, and $Y\in\k_n(Z)$ by induction.  Hence,
  $Y\in\k_n(\k_2(X))=\k_{n+1}(X)$.
\end{proof}

\begin{proof}[Proof of Theorem~\ref{T:kappa}]

  First, we show that $\mathbin{\k}$ is an equivalence relation.  It
  is clear that $\mathbin{\k}$ is reflexive.  If $X\mathbin{\k} Y$ and
  $Y\mathbin{\k} Z$, then $X\in\cC(Y)$ and $Y\in\cC(Z)$.  Since
  $\cC(Z)$ is complete and contains $Y$, $\cC(Y)\subset\cC(Z)$, so
  $X\in\cC(Z)$, i.e., $X\mathbin{\k} Z$.  Finally, if $X\mathbin{\k}
  Y$, then Proposition~\ref{P:closure} implies that $X\in\k_n(Y)$ for
  some $n$.  By the lemma, $Y\in\k_n(X)\subset \k(X)$, and another
  application of Proposition~\ref{P:closure} gives $Y\mathbin{\k} X$.

  Next, suppose that $X\z Y$.  Then $X,Y\subset\prod_{i=1}^s X_i$ for
  some $X_i$'s, so $X\mathbin{\k} Y$.  Since $\mathbin{\k}$ is an
  equivalence relation, it follows that $\z^*\subset \mathbin{\k}$.

Conversely, assume that $X\mathbin{\k} Y$, say $X\in\k_{n+1}(Y)$.  Set
$X_0=X$.  We recursively construct $X_j\in \k_{n+1-j}(Y)$ for $0\le
j\le n$ satisfying $X_j\mathbin{\k} X_{j+1}$.  Choose $X_1\in\k_n(Y)$
such that $X,X_1\subset\prod_{i=1}^{s_1} Y_{1,i}$ for some
$Y_{1,i}$'s.  This means that $X\z X_1$.  Suppose that we have
constructed the desired atoms up through $X_r$ with $r<n$.  Again, we
can choose $X_{r+1}\in \k_{n-r}(Y)$ satisfying
$X_r,X_{r+1}\subset\prod_{i=1}^{s_{r+1}} Y_{r+1,i}$ for some
$Y_{r+1,i}$'s, and this gives $X_r\z X_{r+1}$.  Note that
$X_n\in\k_1(Y)=\{Y\}$, i.e., $X_n=Y$.  We conclude that $X\z^* Y$ as
desired.

\end{proof}

\begin{cor} For any $E\subset \cQ(R)$ nonempty,
  $\psi^{-1}(\psi(E))=\cup_{X\in E} \cC(X)=\cC(E)$.  In particular, the
  $\z^*$-equivalence class of $X$ is $\cC(X)$.
\end{cor}
\begin{proof}  The set  $\psi^{-1}(\psi(E))$ consists of those atoms
  equivalent to an atom in $E$, hence is the union of the $\z^*=\mathbin{\k}$
  equivalence classes of atoms in $E$.  This gives the first
  equality.  The second follows immediately from the fact that a union
  of closed subsets is closed.
\end{proof}

To proceed further, we need to impose additional conditions on $R$.
\begin{prop} \label{P:replaw}\mbox{}\begin{enumerate} \item If $R$ is
    reproducible, then for all $X\in\cQ(R)$, $\cC(X)=\varpi_R\circ
    X=X\circ \varpi_R$.  In particular, the subhypergroupoid
    $\varpi_R$ satisfies the reproductive law.
  \item If $R$ is weakly reproducible, then $\varpi_R$ satisfies the
    weak reproductive law.
  \end{enumerate}
\end{prop}
\begin{proof} First, assume that $R$ is reproducible.  Suppose that
  $Y\in\cC(X)$, so $Y\z^* X$.  By reproducibility, there exist $U$
  such that $Y\subset XU$, i.e., $Y\in X\circ U$.  Hence,
  $\psi(Y)=\psi(X)\psi(U)$, so $\psi(U)=1$.  This shows that
  $U\in\varpi_R$, giving $Y\in X\circ\varpi_R$.  Conversely, if $Z\in
  X\circ\varpi_R$, then $\psi(Z)=\psi(X)$.  This means that
  $Z\in\cC(X)$.  The equality $\cC(X)=\varpi_R\circ X$ is proved in
  the same way.  When $X\in\varpi_R$, the first statement says that
  $\varpi_R=\varpi_R\circ X=X\circ\varpi_R$, which is the reproductive
  law.

If $R$ is weakly reproducible, then the argument given above (with
$Y=k$ and $X\in\varpi_R$) shows that there exists $U,V\in\varpi_R$
such that $k\subset U\circ X\cap X\circ V$ as desired.
\end{proof}

Given $Z\in\cQ(R)$, define \begin{equation*} M(Z)=\{X\in\cQ(R)\mid
  \exists Y_1,\dots,Y_s\in\cQ(R)\text{ such that } X,Z\subset
  \prod_{i=1}^s Y_i\}.
\end{equation*}

\begin{lem}  If $R$ is reproducible (resp. weakly reproducible), then
  $M(Z)$ is a complete part for all $Z$ (resp. for $Z=k$).
\end{lem}
\begin{proof} Assume that $R$ is reproducible.  Take $Y\in M(Z)$, so
  $Y,Z\subset \prod_{i=1}^s Y_i$ for some $Y_i$'s.  Suppose that
  $Y\subset \prod_{j=1}^n Z_j$.  By reproducibility, choose $V,W$ such
  that $Z\subset YV$ and $Z_n\subset WZ$.  Now, suppose that $X\subset
  \prod_{j=1}^n Z_j$ also.  Then \begin{equation*}X\subset
    \prod_{j=1}^n Z_j\subset \left(\prod_{j=1}^{n-1}
      Z_j\right)WZ\subset \left(\prod_{j=1}^{n-1} Z_j\right)WYV\subset
    \left(\prod_{j=1}^{n-1} Z_j\right)W\left(\prod_{i=1}^s
      Y_i\right)V.
\end{equation*}
On the other hand,
\begin{equation*} Z\subset YV\subset \left(\prod_{j=1}^n
    Z_j\right)V\subset \left(\prod_{j=1}^{n-1} Z_j\right) WZV\subset
  \left(\prod_{j=1}^{n-1} Z_j\right)W\left(\prod_{i=1}^s Y_i\right)V.
\end{equation*}
Thus, $Y\in M(Z)$, so $M(Z)$ is complete.

  If $R$ is weakly reproducible, then the same argument works with
  $Z=k$.  Indeed, we need only set $W=Z_n$ and use weak
  reproducibility to choose $V$ such that $k\subset YV$.
\end{proof}

\begin{cor}\label{Cor:m} If $R$ is reproducible (resp. weakly reproducible), then
  $M(Z)=\varpi_R$ for any $Z\in\varpi_R$ (resp. for $Z=k$).
\end{cor}
\begin{proof}  Suppose that $Z\in\varpi_R$.  If
  $X\in M(Z)$, then $X\z Z$ by definition, so $X\in\varpi_R$.  Thus,
  $M(Z)\subset \varpi_R$.  Conversely, the lemma shows that, under the
  hypothesis on $R$, $M(Z)$ is a complete subset containing $Z$, so
  $\cC(Z)=\varpi_R\subset M(Z)$.
\end{proof}

\begin{thm}\label{T:trans} \mbox{}\begin{enumerate}\item If $R$ is reproducible,
    then $\z$ is transitive.
\item If $R$ is weakly reproducible, then $X\z^* k$ implies that $X\z k$.
\end{enumerate}
\end{thm}

\begin{proof} Assume that $R$ is reproducible, and take $X\z^* Y$.  By
  Proposition~\ref{P:replaw}, there exists $U\in\varpi_R$ such that
  $Y\in X\circ U$.  Since $M(k)=\varpi_R$, Corollary~\ref{Cor:m}
  implies the existence of $Y_i$'s such that $U,k\subset\prod_{i=1}^s
  Y_i$.  Thus, $Y\subset XU\subset X\prod_{i=1}^s Y_i\supset Xk=X$, so
  $Y\z X$.  If $R$ is weakly reproducible, the same argument works for
  $Y=k$.
\end{proof}

We are now ready to return to the proof of Theorem~\ref{T:focus}.  We
first state a proposition.

\begin{prop} Let $R$ be weakly reproducible of finite length.  Then
  there exists $X_1,\dots, X_n\in\cQ(R)$ such that $F(R)=\prod_{i=1}^n
  X_i$.
\end{prop}
\begin{proof} First, note that if $k\subset\prod_{i=1}^n X_i$, then
  $\prod_{i=1}^n X_i\subset F(R)$.  Indeed, if $Z\subset\prod_{i=1}^n
  X_i$ is an atom, then $Z\z k$, so $Z\subset F(R)$.  The claim
  follows because $\prod_{i=1}^n X_i$ is the sum of the atoms it
  contains.

  Choose $E=\prod_{i=1}^n X_i$ containing $k$ such that $[E,F(R)]$ has
  minimal length.  If this length is $0$, then $E=F(R)$, so suppose it
  is positive, i.e., $E\subsetneq F(R)$.  Take $Y\in\varpi_R$ such
  that $Y\subsetneq E$.  By Theorem~\ref{T:trans}, $Y\z 1$, so there
  exist $Y_1,\dots, Y_m\in\cQ(R)$ such that $k,Y\subset
  E'=\prod_{j=1}^m Y_j$.  This implies that $E=E k\subset EE'$ and
  $Y=kY\subset EE'$, and the previous paragraph shows that
  $EE'\subset F(R)$. We obtain $E\subsetneq EE'\subset F(R)$,
  contradicting the minimality of the length of $[E,F(R)]$.
\end{proof}

We apply the proposition to prove part~\eqref{repro} of the theorem.
Indeed, if $Z\in\cQ(R)$ and $Z\subset F(R)=\prod_{i=1}^n X_i$, then
$Z\z 1$ by definition.  This means that $Z\in\varpi_R$ as desired.

Finally, we prove part~\eqref{part4}.  Suppose that $R$ is
reproducible of finite length, and $X,Y\subset F(R)$ are atoms.  By
part~\eqref{repro}, $X,Y\in\varpi_R$.  By hypothesis, there exists
$Z\in \cQ(R)$ such that $Y\in X\circ Z$.  Since
$1=\psi(Y)=\psi(X)\psi(Z)=\psi(Z)$, $Z\subset F(R)$ and so
$\cQ([0,F(R)])=X\circ\cQ[0,F(R)])$.  Similarly, one shows
$\cQ([0,F(R)])=\cQ[0,F(R)])\circ X$, so $\cQ([0,F(R)])$ is
reproducible.  The same argument applies when $R$ is weakly
reproducible; here, one takes $Y=k$.  This completes the proof of
Theorem~\ref{T:focus}.

\section{Applications to $G$-algebras}\label{S:Galgs}

In this section, we apply our results on atomistic semirings to
subrepresentation semirings.  We assume throughout that $A$ is a
$G$-algebra which is completely reducible as a representation.  We
write $Q_G(A)$ (resp. $F_G(A)$) instead of $Q(S_G(A))$ (resp.
$F(S_G(A)))$.  We will now be able to generalize our earlier results
on multiplicity-free $G$-algebras.

\begin{prop}\label{P:trivone} Let $A$ be a $G$-algebra in which the trivial
  representation has multiplicity one.  Then $A$
 has no proper, nontrivial one-sided
  invariant ideals if and only if $S_G(A)$ is weakly reproducible.
\end{prop}
\begin{rmk} Note that both conditions imply that $S_G(A)$ is entire.
  Indeed, if the condition on invariant ideals holds, then the
  argument given in the proof of Proposition~\ref{Prop:mf} shows that
  $S_G(A)$ is entire.  The analogous statement for weak reproducibility was
  shown in Remark~\ref{R:entire}.
\end{rmk}

\begin{proof}
  Assume that $A$ has no proper, nontrivial invariant ideals.  Fix
  $X\in\cQ_G(A)$, and express $A$ as a direct sum of irreducible
  subrepresentations $A=\bigoplus_{i\in I} Y_i$.  The subspace $AX$ is
  a nonzero invariant left ideal, so we obtain $A=AX=\sum_{i\in I}
  XY_i$.  The trivial representation must accordingly be an
  irreducible component of some $XY_j$.  The fact that the trivial
  representation has multiplicity one in $A$ implies that $k\subset
  XY_j$ as desired.  Similarly, since $A=XA$, there exists $Y_l$ such
  that $k\subset Y_lX$.

Conversely, suppose that $0\ne L\ne A$ is an invariant left ideal.
Let $X\subset L$ be an irreducible submodule.  For any $Y\in\cQ_G(A)$,
we have $YX\subset L$; since $k\cap L=0$, $S_G(A)$ is not weakly
reproducible.  A similar argument works for right ideals.
\end{proof}

\begin{cor} The atomistic semiring $S_G(A)$ is weakly reproducible of
  finite length if
\begin{enumerate}\item $A$ is a finite Galois extension of $k$ and $G$ is the
  Galois group; or
\item $A=\End(V)$ is a finite-dimensional $G$-algebra whose underlying
  projective representation $V$ is irreducible.
\end{enumerate}
\end{cor}
\begin{proof}  Schur's lemma shows that $\End(V)$ contains the trivial
  representation with multiplicity one, and the statement about
  invariant ideals was proved in~\cite[Theorem 5.2]{S1}.  The analogous
  verifications for the other case are obvious.
\end{proof}

We are thus able to define invariants for any $G$-algebra satisfying
the conditions of Proposition~\ref{P:trivone}, without our earlier
assumption that the $G$-algebra is multiplicity-free.  In particular,
our results determine two new sequences of invariants associated to
any irreducible projective representation, namely, the condensation
groups $Q^n_G(\End(V))$ and the focal subalgebras $F^n_G(\End(V))$.

The focal subalgebras $F_G^n(A)$ are a decreasing sequence of
invariant subalgebras (i.e., subalgebras which are also
subrepresentations) of $A$.  This is particularly interesting for
$A=\End(V)$ with $V$ irreducible and $k$ algebraically closed because
in this case, there is a complete classification of such invariant
subalgebras in terms of representation-theoretic data~\cite[Theorem
3.23]{S1}.

For the rest of the paper, we assume that either $G$ is finite and $k$
is algebraically closed of characteristic zero or $G$ is a compact
group and $k=\C$.  We let $V$ be an irreducible (linear)
representation of $G$, and set $A=\End(V)$.  (We make these assumptions
on $G$ and $k$ to guarantee complete reducibility of $\End(V)$; the
classification of invariant subalgebras described below holds in
general.)

An invariant subalgebra of $A$ is determined by data consisting of a
quadruple $(H,W,U,U')$; here, $H$ is a finite index subgroup of $G$,
$W$ is a linear representation of $H$ such that $V=\Ind_H^G(W)$, and
$U,U'$ are a pair of projective representations of $H$ such that
$W\cong U\otimes U'$.  More precisely, there is a bijection between
invariant algebras and equivalence classes of such quadruples under
conjugation by $G$.  In particular, there are a finite number of
invariant subalgebras.

Given such a quadruple $(H,W,U,U')$, we construct the corresponding
invariant subalgebra as follows: Let $g_1\mathbin{=}e,g_2,\dots,g_n$ be a left
transversal for $H$ in $G$.  This gives a direct sum decomposition
$V=\bigoplus_{i=1}^n g_iW$ and an associated block diagonal invariant
subalgebra $\Ind_H^G(\End(W))\overset{\text{def}}{=}\bigoplus_{i=1}^n
\End(g_iW)$.  As an algebra, this is just the direct product of $n$
copies of $\End(W)$.  Next, the isomorphism $W\cong U\otimes U'$ shows
that the endomorphism algebra factors (as $H$-algebras) into the
tensor product $\End(W)\cong
\End(U)\otimes\End(U')$.  It is now immediate that $\End(U)\otimes k$
is an $H$-invariant subalgebra of $\End(W)$.  Finally, we obtain the
invariant algebra for the quadruple: $\Ind_H^G( \End(U)\otimes k)$.
We remark that the two obvious invariant subalgebras $k$ and
$\End(V)$ correspond to $(G,V,k,V)$ and $(G,V,V,k)$ respectively.

It now follows that the sequence of focal subalgebras associated to
the irreducible representation $V$ gives rise to a sequence of such
quadruples.

The classification of invariant subalgebras can be very helpful for
computing the $F^n_G(\End(V))$.  For example, suppose that $\End(V)$
has no nontrivial invariant subalgebras, so that any irreducible
representation generates $\End(V)$.  In order to show that
$F_G(\End(V))=\End(V)$, it is only necessary to check that
$F_G(\End(V))$ contains a nonscalar matrix.  However, it should be
noted that computing the invariant subalgebras is not necessarily
straightforward.  Even when $G$ is finite, it is not determined by the
character table of $G$.  In general, one needs to know the character
tables of a covering group for every subgroup of $G$ whose index
divides $\dim(V)$.
 
\begin{exam} Let $V$ be the standard representation of $S_3$.  We have
  already seen that $Q^1_G(\End(V))=\Z_2$.  The focal subalgebra
  $F^1_G(\End(V))=\C\oplus \s$ is isomorphic to $\C\oplus\C$ as an
  algebra; it comes from the quadruple $(A_3,\chi,\chi,\C)$, where
  $\chi$ is either nontrivial character of $A_3$.  Since $\C$ and $\s$
  are not $\z^*$-equivalent in $\cQ_G(F^1_G(\End(V)))$, we have
  $Q^2_G(\End(V))=\Z_2$ and for $m\ge 2$, $F^m_G(\End(V))=\C$
  (corresponding to $(S_3,V,\C,V)$).  Finally, $Q^n_G(\End(V))=1$ for
  $n\ge 3$,
\end{exam}

\begin{exam}\label{E:A4} Let $W$ be the three-dimensional irreducible
  representation of $A_4$.  We will show that $\cQ_{A_4}(\End(W))$ is
  not associative and does not satisfy the reproductive law.

  We have the direct sum decomposition $\End(W)=\C\oplus Z\oplus
  Z'\oplus X\oplus Y$, where $Z$ and $Z'$ correspond to the two
  nontrivial characters of $A_4$ and $X$ and $Y$ are isomorphic to
  $W$.  We can choose a basis for $W$ with respect to which $Y$ (resp.
  $X$) consists of the skew-symmetric (resp.  off-diagonal symmetric)
  matrices and the diagonal $T$ is the direct sum of $\C$, $Z$, and
  $Z'$.  There are an infinite number of atoms isomorphic to $W$,
  parameterized by $[a:b]\in\mathbf{P}^1(\C)$; we set
  \begin{equation*} U_{[a:b]}=\spa\{(a+b)E_{23}+(a-b)E_{32},(a-b)E_{13}+(a+b)E_{31},(a+b)E_{12}+(a-b)E_{21}\}.
  \end{equation*}
  In this notation, $X=U_{[1:0]}$ and $Y=U_{[0:1]}$.

  Let $P=U_{[1:1]}$.  It is easily checked that $P^2=U_{[1:-1]}$.  We
  now calculate that $P\C=PZ=PZ'=P$, $P(P^2)=T$, and
  $PU_{[a:b]}=T\oplus P^2$ for $[a:b]\ne [1:\pm 1]$.  We thus see that
  $\cQ_{A_4}(\End(W))$ does not satisfy the reproductive law; if
  $[a:b]\ne [1:\pm 1]$, there is no $V$ for which $U_{[a:b]}\in P\circ
  V$.  To verify that the associative law does not hold, note that
  $X\in(P\circ P^2)\circ X=\{\C,Z,Z'\}\circ X$.  However,
  $P\circ(P^2\circ X)=P\circ\{\C,Z,Z',P\}=\{P,P^2\}$ does not contain
  $X$.

Since $PX=T\oplus P^2$, we have $\C,Z,Z',P^2\in\varpi$.  Also, $P^2
X=T\oplus P$, so $Q\in\varpi$.  This implies that
$F^n_{A_4}(\End(W))=\End(W)$ for all $n$, and by
Corollary~\ref{C:max}, $Q^n_{A_4}(\End(W))=1$ for all $n$.

The only nontrivial invariant subalgebra of $\End(W)$ is $T$.  (It
corresponds to $(H,\chi,\chi,\C)$, where $H\cong\Z_2\times\Z_2$ is the
subgroup of order $4$ and $\chi$ is any nontrivial character of $H$.)
Thus, one knows that $F_{A_4}(\End(W))=\End(W)$ as soon as one know
that $P^2\in\varpi$.

\end{exam}

We conclude by computing the condensation groups and focal subalgebras
of endomorphism algebras for simple compact Lie groups.
\begin{thm}\label{T:Lie}  Let $V$ be an irreducible representation of the 
  simple compact Lie group $G$.  Then $Q^n_G(\End(V))=1$ and
  $F^n_G(\End(V))=\End(V)$ for all $n$.
\end{thm}
\begin{proof}  If $V=\C$, the statement is trivial.  Any other $V$ has
  dimension at least $2$.  By Corollary~\ref{C:max}, it suffices to show that
  $F_G(\End(V))=\End(V)$.  Moreover, by \cite[Theorem 4.3]{S1}, the
  only proper invariant subalgebra of $\End(V)$ is $\C$.  Hence, we
  need only show that $F_G(\End(V))$ contains a nonscalar matrix.

  Let $\l$ be the highest weight of $V$.  The highest weight of the
  dual representation $V^*$ is $-w_0\l$, where $w_0$ is the longest
  element in the Weyl group.  The representation $\End(V)\cong
  V\otimes V^*$ has a unique irreducible submodule $X$ with highest
  weight $\l-w_0\l$.  We can write down a highest and lowest weight
  vector in $X$ explicitly.  Let $v_\l$ (resp. $w_\l$) be a highest
  (resp. lowest) weight vector in $V$.  (The highest and lowest
  weights are different since $\dim V\ge 2$.)  Extend the set
  $\{v_\l,w_\l\}$ to a basis of weight vectors for $V$, and let
  $v^*_\l, w^*_\l$ be the corresponding dual basis vectors in $V^*$.
  Then $w^*_\l$ (resp. $v^*_\l$) is a highest (resp. lowest) weight
  vector in $V^*$.  It follows that $v_\l\otimes w^*_\l$ (resp.
  $w_\l\otimes v^*_\l$) is a highest (resp. lowest) weight vector in
  $X$.

  Multiplying, we obtain $z=(v_\l\otimes w^*_\l)(w_\l\otimes v^*_\l)=
  v_\l\otimes v^*_\l\in X^2$.  The matrix $z$ has rank one, so is not
  a scalar matrix.  Thus, $X^2\ne \C$.  However, $\tr(z)=1$, so $z$ is
  not orthogonal to $\C$.  This implies that $\C\subset X^2$.  We
  conclude that $\varpi$ contains at least two elements, so
  $F_G(\End(V))\ne \C$.
\end{proof}



\bibliographystyle{pnaplain}

\begin{thebibliography}{99}

\bibitem{Cor} P. Corsini, {\it Prolegomena of hypergroup theory},
  Supplement to Riv. Mat. Pura. Appl., Aviani Editore, Tricesimo,
  1993.

\bibitem{Freni91} D. Freni, {\sl Une note sur le coeur d'un
    hypergroupe et sur la cl\^{o}ture transitive $\beta^*$ de
    $\beta$}, Riv. Mat. Pura Appl. {\bf 8} (1991), 153--156.

\bibitem{Freni02} D. Freni, {\sl A new characterization of the derived
  hypergroup via strongly regular equivalences}, Comm. Algebra {\bf 32} (2002), 3977--3989.

\bibitem{GMS} Y. Grabovsky, G. Milton, and D.~S. Sage, {\sl Exact
    relations for effective tensors of polycrystals: Necessary and
    sufficient conditions}, Comm. Pure. Appl. Math.  {\bf 53} (2000),
  300--353.

\bibitem{GS} Y. Grabovsky and D. S. Sage, {\sl Exact relations for
    effective tensors of polycrystals.  II: Applications to elasticity
    and piezoelectricity}, Arch. Rat. Mech. Anal. {\bf 143} (1998),
  331--356.

\bibitem{KS} N. Kwon and D.~S. Sage, {\sl Subrepresentation semirings
    and an analogue of $6j$-coefficients}, J. Math. Phys. {\bf 49}
  (2008), 063503.

\bibitem{Koskas} M. Koskas, {\sl Groupo\"{i}des, demi-hypergroupes et
    hypergroupes}, J. Math. Pures Appl. {\bf 49} (1970), 155--192.

\bibitem{S1} D.~S. Sage, {\sl  Group actions on central simple
    algebras}, J. Algebra {\bf 250} (2002), 18--43.

\bibitem{S2} D.~S. Sage, {\sl Racah coefficients, subrepresentation
    semirings, and composite materials}, Adv. App. Math {\bf 34}
  (2005), 335--357.

\bibitem{S3} D.~S. Sage, {\sl Quantum Racah coefficients and
    subrepresentation semirings}, J. Lie Theory {\bf 15} (2005),
321--333. 

\bibitem{S4} D.~S. Sage, {\sl Subrepresentation semirings and
    $3nj$-symbols for simply reducible groups}, preprint.

\bibitem{Vou} T. Vougiouklis, {\it Hyperstructures and their
    representations}, Hadronic Press, Inc., Palm Harbor, Fl, 1994.

 \end{thebibliography}

\end{document}